\documentclass[a4paper,11pt,twoside]{amsart}

\usepackage{amsmath}
\usepackage{amsthm}
\usepackage{amsfonts}
\usepackage{amssymb}
\usepackage{mathrsfs}
\usepackage{graphicx}
\usepackage{enumerate}
\usepackage{url}

\usepackage{tikz}

\usepackage[labelformat=simple]{subcaption}

\usetikzlibrary{arrows}
\usetikzlibrary{decorations.markings}

\theoremstyle{plain}
\newtheorem{lemma}{Lemma}[section]
\newtheorem{prop}[lemma]{Proposition}
\newtheorem{theo}[lemma]{Theorem}
\newtheorem{coro}[lemma]{Corollary}
\theoremstyle{remark}
\newtheorem{rem}[lemma]{Remark}

\theoremstyle{definition}
\newtheorem{definition}[lemma]{Definition}
\newtheorem{ex}[lemma]{Example}

\newcommand{\op}{\textnormal{op}}
\newcommand{\id}{\mathrm{Id}}
\newcommand{\F}{\mathscr{F}}
\newcommand{\N}{\mathbb{N}}
\newcommand{\tq}{\;/\;}
\renewcommand{\S}{\mathcal{S}}
\newcommand{\fintop}{\text{\bf FinTop}}
\newcommand{\fintopcero}{\text{\bf FinTop}_{\bf 0}}
\DeclareMathOperator{\im}{Im}

\begin{document}
\title[Cofibrations between finite topological spaces]{A combinatorial characterization of Hurewicz cofibrations between finite topological spaces}

\author{Nicol\'as Cianci}
\address{Facultad de Ciencias Exactas y Naturales \\ Universidad Nacional de Cuyo and CONICET \\ Mendoza, Argentina.}
\email{nicocian@gmail.com}

\author{Miguel Ottina}
\address{Facultad de Ciencias Exactas y Naturales \\ Universidad Nacional de Cuyo \\ Mendoza, Argentina.}
\email{mottina@fcen.uncu.edu.ar}



\thanks{Research partially supported by grants M015 (2013--2015) and M044 (2016--2018) of SeCTyP, UNCuyo. The first author was also partially supported by a CONICET doctoral fellowship.}

\begin{abstract}
We characterize the Hurewicz cofibrations between finite topological spaces, that is, the continuous functions between finite topological spaces that have the homotopy extension property with respect to all topological spaces. In particular, we show that cofibrations between connected non-empty finite topological spaces are homotopy equivalences.

As a consequence of our characterization, we obtain a simple algorithm capable of determining whether a given continuous function between finite topological spaces is a cofibration.
\end{abstract}

\maketitle

\section{Introduction}

Finite topological spaces are naturally endowed with a very interesting combinatorial flavour that is based in the well-known bijective correspondence between topologies in a finite set $X$ and preorders in $X$ given by Alexandroff \cite{alexandroff1937diskrete}. Under this bijection T$_0$ topologies correspond to partial orders. Moreover, any finite T$_0$--space is weak homotopy equivalent to the geometric realization of the order complex of its associated poset \cite{mccord1966singular}. These results permit the study of homotopy properties of the order complex of a finite poset by means of its associated finite T$_0$--space, which has proved to be very fruitful \cite{barmak2011algebraic}. Indeed, the theory of finite topological spaces gives interesting tools to study posets and polyhedra.

A natural question to pose in the homotopy theory of finite spaces is which maps between finite spaces are Hurewicz cofibrations. For example, it is easy to prove that if $\S=\{0,1\}$ is the Sierpinski space (where $\{0\}$ is the only non-trivial open set of $\S$), then the inclusion $\{0\}\hookrightarrow \S$ is a (Hurewicz) cofibration but the inclusion $\{1\}\hookrightarrow \S$ is not. The main result of this article can be seen as both an explanation and an extensive generalization of this fact, and allows us to completely characterize the cofibrations between finite (not necessarily T$_0$) spaces. In particular, we obtain the unexpected results that cofibrations between connected non-empty finite spaces are homotopy equivalences and that closed cofibrations between connected non-empty finite spaces are homeomorphisms.

Our characterization of cofibrations between finite spaces is purely combinatorial and surprisingly simple: if $X$ is a connected finite topological space and $A\subseteq X$ is a non-empty subspace then the inclusion map $i\colon A\hookrightarrow X$ is a cofibration if and only if there exists a retraction $r\colon X \to A$ of $i$ such that $ir\leq \id_X$. Moreover, we prove that for finite T$_0$--spaces such a retraction exists if and only if the subspace $A$ is, in our terminology, a \emph{dbp--retract} of $X$, that is, the subspace $A$ can be obtained from $X$ by successively removing \emph{down beat points} (which are points that have exactly one lower cover). In addition, we prove that a inclusion map $i\colon A\hookrightarrow X$ as above is a cofibration if and only if the induced map between the Kolmogorov quotients of $A$ and $X$ is a cofibration. As a corollary of these results we obtain a simple algorithm for determining whether a continuous map between finite topological spaces is a cofibration.

Dbp--retracts turn out to have very interesting properties which are studied in section \ref{sec_bpr} of this article. For example, for any finite T$_0$--space $X$, the set $\Omega(X,\varnothing)$ of dbp--retracts of $X$ is in one-to-one correspondence with the set 
\begin{displaymath}
	\F(X,\varnothing)=\{f\in X^{X}/f\leq\id_X\text{ and }f^{2}=f\},
\end{displaymath}
where $X^{X}$ is the space of continuous functions from $X$ to itself.

The set $\F(X,\varnothing)$ is not closed under compositions, but is closed under the operation $*$ defined by $f*g=(fg)^{\infty}$, where $f^{\infty}$ denotes the composition of $f$ with itself a sufficiently large amount of times. It is not hard to see that $*$ turns $\F(X,\varnothing)$ into an abelian monoid, which is a consequence of the fact that the set $\Omega(X,\varnothing)$ is closed under intersections. In particular, every finite T$_0$--space $X$ has a unique minimal dbp--retract. This means that successively removing down beat points of $X$ furnishes a space without down beat points which is independent of the order in which the down beat points of $X$ (and the successively obtained subspaces) are removed.

\section{Preliminaries}
For every finite topological space $X$ and for every $x\in X$, there is a minimal open set that contains $x$, which will be denoted by $U^{X}_x$ (or by $U_x$, if the space $X$ is understood). The space $U^{X}_x-\{x\}$ will be denoted by $\widehat{U}^{X}_x$.

If $X$ is a finite topological space, a preorder $\leq$ is defined on $X$ as follows: for $x,y\in X$,
\begin{displaymath}
	 x\leq y \text{ if and only if } U_x\subseteq U_y.   
\end{displaymath}
Note that $\leq$ is nothing but the specialization preorder in $X$, that is, $x\leq y$ if and only if $y$ is in the closure of $\{x\}$. Thus, $\leq$ is an order if and only if $X$ is a T$_0$--space.
On the other hand, if $X$ is a finite preordered set, the lower sets of $X$ form a topology on $X$. It is well known that these constructions are mutually inverse and provide a functorial bijective correspondence between finite spaces and finite preordered sets that restricts to a bijective correspondence between finite T$_0$--spaces and finite posets \cite{alexandroff1937diskrete}.  
Hence, from now on, every finite space will be considered as a finite preordered set, every finite T$_0$--space will be considered as a finite poset and every continuous function between finite spaces will be considered as an morphism of preordered sets without further notice. 

Let $X$ be a finite space considered as a preordered set with preorder $\leq$. The preorder $\leq^{\op}$ is defined on $X$ by 
	\begin{displaymath}
		x\leq^{\op} y \text{ if and only if } y\leq x
	\end{displaymath}
for $x,y\in X$. The preorder $\leq^{\op}$ induces a topological space $X^{\op}$ with the same underlying set as $X$ but whose open sets are the closed sets of $X$. A continuous function $f\colon X\to Y$ between finite spaces can be regarded as  a continuous function $f^{\op}\colon X^{\op}\to Y^{\op}$.  

Now, if $X$ is a finite space, an equivalence relation $\sim$ on $X$ is defined by
\begin{displaymath}
	 x\sim y \text{ if and only if } U_x=U_y.   
\end{displaymath}
Hence, $x\sim y$ if and only if $x\leq y$ and $y\leq x$. We will denote $X/\sim$ by $X_0$ and the canonical quotient map $X\to X_0$ by $q_X$. Note that every function $j\colon X_0\to X$ such that $j(q_X(x))\sim x$ for all $x\in X$ is continuous (since $jq_X$ is order preserving) and hence a section of $q_X$. 
It is easy to see that $X_0$ is a T$_0$--space and that $x\leq x'$ in $X$ if and only if $q_X(x)\leq q_X(x')$ in $X_0$ \cite{stong1966finite,barmak2011algebraic}.

Observe that a continuous function $f\colon X\to Y$ induces a unique continuous function $f_0\colon X_0\to Y_0$ such that $f_0q_X=q_Yf$. The assignments $X\mapsto X_0$ and $f\mapsto f_0$ define a functor $T$ from the category $\fintop$ of finite topological spaces (and continuous functions) to the category $\fintopcero$ of finite T$_0$--spaces (and continuous functions). The functor $T$ is left-adjoint to the inclusion functor $i\colon \fintopcero\to\fintop$.

For a finite T$_0$--space $X$ and elements $x,y$ in $X$, we will write $x<y$ when $x\leq y$ and $x\neq y$. We will also write $x\geq y$ if $y\leq x$ and $x>y$ if $y<x$.

If $X$ and $Y$ are finite spaces, the space of continuous functions from $X$ to $Y$ (equipped with the compact-open topology) will be denoted by $Y^{X}$. 
Note that, in this case, $Y^{X}$ is a finite topological space, and for any two continuous functions $f,g\in Y^{X}$ we have that $f\leq g$ in $Y^{X}$ if and only if $f(x)\leq g(x)$ for every $x$ in $X$ \cite{stong1966finite}.
Thus, for every $f,g\in Y^{X}$, we have that $f\leq g$ in $Y^{X}$ if and only if $f_0\leq g_0$ in $Y_0^{X_0}$.

In \cite{stong1966finite}, R. E. Stong also proved that a finite space is connected if and only if it is connected when considered as a preordered set, and if and only if it is path-connected. In the same work, it is proved that if $X$ and $Y$ are finite topological spaces and $f,g\in Y^{X}$ are such that $f\leq g$, then $f$ is homotopic to $g$ relative to the set $\{x\in X:f(x)=g(x)\}$.

Let $X$ be a finite T$_0$--space and let $x\in X$. As in \cite{barmak2011algebraic}, we say that $x$ is a \emph{down beat point} of $X$ if the set $\widehat{U}^{X}_x=\{z\in X:z<x\}$ has a maximum, or equivalently, if the element $x$ has exactly one lower cover.

Similarly, we say that $x$ is an \emph{up beat point} of $X$ if the set $\{z\in X:z>x\}$ has a minimum.

We say that $x$ is a \emph{beat point} of $X$ if it is either a down beat point of $X$ or an up beat point of $X$.

\begin{rem}[Stong, \cite{stong1966finite}]
	Let $X$ be a finite T$_0$--space and let $x$ be a down beat point of $X$. Let $i\colon X-\{x\}\to X$ be the inclusion and let $r\colon X\to X-\{x\}$ be the function defined by
	\begin{displaymath}
		r(z)=\begin{cases}z&\text{if $z\neq x$,}\\\max(\widehat{U}_x)&\text{if $z=x$.}\end{cases}
	\end{displaymath}
	It is easy to see that $r$ is continuous, that $ri=\id_{X-\{x\}}$ and that $ir\leq \id_X$. In particular, $X-\{x\}$ is a strong deformation retract of $X$.
	\label{rem_retract_beat_points}
\end{rem}

We will need the following well-known result, which is easy to prove. 
\begin{lemma}
	Let $X$ be a connected finite space, let $Y$ be a T$_1$--space and let $f\colon X\to Y$ be a continuous function. Then $f$ is a constant map.
	\label{lemma_cte}
\end{lemma}

\section{Bp--retracts of finite T$_0$--spaces}
\label{sec_bpr}

In this section, we introduce the concepts of \emph{dbp--retracts} and \emph{ubp--retracts} of finite T$_0$--spaces and prove some properties that will be needed in section \ref{sec-cofibrations}.  
\begin{definition}
	Let $X$ be a finite T$_0$--space and let $A\subseteq X$. We will say that $A$ is a \emph{dbp--retract} (resp.\ \emph{ubp--retract}) of $X$ if $A$ can be obtained from $X$ by successively removing down beat points (resp.\ up beat points), that is, if there exist $n\in\N_0$ and a sequence $X=X_0\supseteq X_1\supseteq\cdots\supseteq X_n=A$ of subspaces of $X$ such that, for all $i\in\{1,\ldots,n\}$, the space $X_i$ is obtained from $X_{i-1}$ by removing a single down beat point (resp.\ up beat point) of $X_{i-1}$.
	
	We will say that $A$ is a \emph{bp--retract} of $X$ if $A$ is either a dbp--retract or a ubp--retract of $X$.
\end{definition}

In particular, $X$ is a dbp--retract and a ubp--retract of itself.

\begin{rem}
	\label{rem_dbpr_of_dbpr_is_dbpr}
	If $Y$ is a finite T$_0$--space, $X$ is a dbp--retract of $Y$ and $A$ is a dbp--retract of $X$, then $A$ is a dbp--retract of $Y$.
\end{rem}

\begin{rem}
	\label{rem_mnl_X_in_A}
	Note that if $X$ is a finite T$_0$--space and $A$ is a dbp--retract of $X$ then the minimal points of $X$ are contained in $A$, since they cannot be down beat points of any subspace of $X$. In particular, $A$ must be dense in $X$. 
\end{rem}
\begin{rem}
	\label{rem_ubp_iff_dbp_op}
	If $X$ is a finite T$_0$--space, then a subspace $A$ of $X$ is a ubp--retract of $X$ if and only if $A^{\op}$ is a dbp--retract of $X^{\op}$.
\end{rem}
	In the rest of this section, we will prove several results for dbp--retracts of finite T$_0$--spaces. Similar results hold for ubp--retracts by \ref{rem_ubp_iff_dbp_op}.
	
\begin{theo} \label{theo_dbpr_equivalences}
	Let $X$ be a finite T$_0$--space, let $A$ be a subspace of $X$ and let $i\colon A\to X$ be the inclusion map. Then, the following propositions are equivalent:
	\begin{enumerate}[(1)]
		\item $A$ is a dbp--retract of $X$.
		\item There exists a continuous function $f\colon X\to X$ such that $f\leq\id_X$, $f^{2}=f$ and $f(X)=A$.
		\item There exists a unique continuous function $f\colon X\to X$ such that $f\leq\id_X$, $f^{2}=f$ and $f(X)=A$.
		\item There exists a retraction $r\colon X\to A$ of $i$ such that $ir\leq\id_X$.
		\item There exists a unique retraction $r\colon X\to A$ of $i$ such that $ir\leq\id_X$.
	\end{enumerate}
\end{theo}
\begin{proof} We will prove that $(1)\Rightarrow (4)\Rightarrow (5)\Rightarrow (3)\Rightarrow (2)\Rightarrow (1)$.

	The implication $(1)\Rightarrow (4)$ follows from \ref{rem_retract_beat_points}.
	
	Next, we will show that $(4)\Rightarrow (5)$.
	For $k=1,2$, let $r_k\colon X\to A$ be a retraction of $i$ such that $ir_k\leq\id_X$. We wish to show that that $r_1=r_2$. Let $x\in X$. 
	Since $r_2(x)\in A$ and $r_2(x)\leq x$ it follows that $r_2(x)=r_1r_2(x)\leq r_1(x)$. Similarly, $r_1(x)\leq r_2(x)$ and thus, $r_1(x)=r_2(x)$. Therefore, $r_1=r_2$.
	
	We will now prove that $(5)\Rightarrow (3)$. Suppose that there exists a unique retraction $r\colon X\to A$ of $i$ such that $ir\leq\id_X$. It is clear that $(ir)^{2}=ir$ and that $ir(X)=A$. Thus, we have proved that $(5)$ implies $(2)$.
	
	Now, for $k=1,2$, let $f_k\colon X\to X$ be a continuous function such that $f_k\leq \id_X$, $f_k^{2}=f_k$ and $f_k(X)=A$ and let $r_k\colon X\to A$ be the range restriction of $f_k$. Since $f_k^{2}=f_k$ and $f_k\leq\id_X$, it is clear that $r_k$ is a retraction of $i$ such that $ir_k\leq\id_X$ for $k=1,2$. Hence, $r_1=r_2$. It follows that $f_1=f_2$.
	
	The implication $(3)\Rightarrow (2)$ is clear. 
	
	Next, we will show that $(2)\Rightarrow (1)$. Let $f\colon X\to X$ be a continuous function such that $f\leq\id_X$, $f^{2}=f$ and $f(X)=A$.	
	
	Let $W=\{x\in X:f(x)<x\}$. If $W=\varnothing$ then $X=f(X)=A$ and the result follows. 
	
	Suppose that $W\neq \varnothing$.
	Let $x_0$ be a minimal element of $W$. Since $x_0\in W$, $f(x_0)<x_0$. We claim that $f(x_0)=\max \widehat U_{x_0}$. Indeed, if $x_1<x_0$ then $x_1\not\in W$ and therefore $x_1=f(x_1)\leq f(x_0)$. It follows that $x_0$ is a down beat point of $X$.
	
	Let $X'=X-\{x_0\}$. Thus, $X'$ is obtained from $X$ by removing a single down beat point. Since $f\circ f=f$ it follows that $x_0\not\in f(X)$, and therefore, we can restrict $f$ to a function $f'\colon X'\to X'$. It is clear that $f'\leq \id_{X'}$ and that $f'\circ f'=f'$. Moreover, $\im f'=\im f=A$ since $f(x_0)=f(f(x_0))\in f(X')$. The result follows by an inductive argument.
\end{proof}
\begin{rem}
	\label{rem_f_r}
	Let $X$ be a finite T$_0$--space, let $A$ be a dbp--retract of $X$ and let $i\colon A\to X$ be the inclusion map. From the proof of \ref{theo_dbpr_equivalences}, it is clear that the unique continuous function $f\colon X\to X$ such that $f\leq\id_X$, $f^{2}=f$ and $f(X)=A$ and the unique retraction $r\colon X\to A$ of $i$ such that $ir\leq\id_X$ are related by $f=ir$. Equivalently, $r$ is the range restriction of $f$ to its image $A$.
\end{rem}
\begin{theo}
	\label{theo_Ux_cap_A}
	Let $X$ be a finite T$_0$--space and let $A$ be a subspace of $X$. Then, $A$ is a dbp--retract of $X$ if and only if $U^{X}_x\cap A$ has a maximum for every $x\in X$. Equivalently, $A$ is a dbp--retract of $X$ if and only if $U^{X}_x\cap A$ has a maximum for every $x\in X-A$.
\end{theo}
\begin{proof}
	Suppose that $A$ is a dbp--retract of $X$. Let $r\colon X\to A$ be the only retraction of $i$ such that $ir\leq \id_X$ and let $x\in X$. It is clear that $r(x)\in U^{X}_x\cap A$. Now, if $y\in U^{X}_x\cap A$ then $y=r(y)\leq r(x)$ since $y\leq x$. Therefore, $r(x)$ is the maximum of $U^{X}_x\cap A$.
	
	Now, suppose that $U^{X}_x\cap A$ has a maximum for every $x\in X-A$.
	Note that $a$ is the maximum of $U^{X}_a\cap A$ for every $a\in A$. 
	Let $r\colon X\to A$ be defined by $r(x)=\max (U^{X}_x \cap A)$. If $x\leq x'$ then $U^{X}_x\subseteq U^{X}_{x'}$ and hence $r(x)\leq r(x')$. It follows that $r$ is continuous. It is clear that $ri=\id_A$ and that $ir\leq\id_X$. Hence, by \ref{theo_dbpr_equivalences} $A$ is a dbp--retract of $X$.
\end{proof}
\begin{prop}
	Let $Y$ be a finite T$_0$--space, let $X\subseteq Y$ and let $A\subseteq X$ be a dbp--retract of $Y$. Then $A$ is a dbp--retract of $X$.
	\label{prop_dbpr_of_Y_and_sub_of_X_is_dbpr_of_X}
\end{prop}
\begin{proof}
	Let $i\colon A\to X$ and $j\colon X\to Y$ be the inclusion maps and let $r\colon Y\to A$ be a continuous function such that $rji=\id_A$ and $jir\leq \id_Y$. Then $rj\colon X\to A$ satisfies $rji=\id_A$ and $irj\leq \id_X$, and thus, $A$ is a dbp--retract of $X$ by \ref{theo_dbpr_equivalences}.
\end{proof}
\begin{coro}
	Let $Y$ be a finite T$_0$--space, let $X$ be a dbp--retract of $Y$ and let $A\subseteq X$. Then $A$ is a dbp--retract of $X$ if and only if $A$ is a dbp--retract of $Y$.
	\label{coro_dbpr}
\end{coro}
\begin{proof}
	Immediate from 	\ref{rem_dbpr_of_dbpr_is_dbpr} and \ref{prop_dbpr_of_Y_and_sub_of_X_is_dbpr_of_X}.
\end{proof}
\begin{prop}
	Let $X$ be a finite T$_0$--space and let $A_1$ and $A_2$ be two dbp--retracts of $X$. For $k=1,2$, let $i_k\colon A_k\to X$ be the inclusion and let $r_k\colon X\to A_k$ be the unique retraction of $i_k$ such that $i_k r_k\leq \id_X$. Then $A_1\subseteq A_2$ if and only if $i_1 r_1\leq i_2 r_2$.
	\label{prop_A_subseteq_iff_f_leq}
\end{prop}
\begin{proof}
	Suppose that $A_1\subseteq A_2$ and let $i\colon A_1\to A_2$ be the inclusion. By \ref{prop_dbpr_of_Y_and_sub_of_X_is_dbpr_of_X}, $A_1$ is a dbp--retract of $A_2$. Hence, there exists a retraction $r$ of $i$ such that $ir\leq \id_{A_2}$. Now, since $i_2 i=i_1$, it follows that $rr_2i_1=rr_2i_2i=\id_{A_1}$ and $i_1rr_2=i_2irr_2\leq i_2r_2\leq \id_X$. By \ref{theo_dbpr_equivalences}, $rr_2=r_1$. Thus, $i_1r_1=i_2irr_2\leq i_2r_2$.
	
	Now, suppose that $i_1r_1\leq i_2r_2$ and let $w\in A_1$. Then $w=i_1r_1(w)\leq i_2r_2(w)\leq w$ and it follows that $i_2r_2(w)=w$. Hence, $w\in A_2$. The result follows.
\end{proof}
\begin{coro}
	\label{coro_subseteq_iff_f1_leq_f2}
	Let $X$ be a finite T$_0$--space and, for $k=1,2$, let $f_k\colon X\to X$ be a continuous function such that $f_k\circ f_k=f_k$ and $f_k\leq \id_X$. Then $f_1\leq f_2$ if and only if $f_1(X)\subseteq f_2(X)$.
\end{coro}
\begin{proof}
	By \ref{theo_dbpr_equivalences}, $f_1(X)$ and $f_2(X)$ are dbp--retracts of $X$. The result follows from \ref{rem_f_r} and \ref{prop_A_subseteq_iff_f_leq}.
\end{proof}

\begin{definition}
	Let $X$ be a finite T$_0$--space and let $A$ be a subspace of $X$. We define
		\begin{displaymath}
		\F(X,A)=\{f\in X^{X}:f\leq\id_X,\ f^{2}=f,\text{ and }A\subseteq f(X)\}
	\end{displaymath}
	and
	\begin{displaymath}
		\Omega(X,A)=\{W\subseteq X:\text{$W$ is a dbp--retract of $X$ and $A\subseteq W$}\}.
	\end{displaymath}
	
	The set $\F(X,A)$ will be considered as a subposet of $X^{X}$ and the set $\Omega(X,A)$ will be considered as a poset with the order given by set inclusion. 
\end{definition}
\begin{rem}
	Let $X$ be a finite T$_0$--space and let $X_m$ denote the set of minimal elements of $X$. By \ref{rem_mnl_X_in_A}, it is clear that $\Omega(X,\varnothing)=\Omega(X,A)$ for every $A\subseteq X_m$.
\end{rem}

\begin{prop}
	Let $X$ be a finite T$_0$--space and let $A$ be a subspace of $X$. Then $\F(X,A)$ is order isomorphic to $\Omega(X,A)$.
	\label{prop_Omega_iso_F}
\end{prop}
\begin{proof}
	By \ref{theo_dbpr_equivalences}, there is a bijection $\varphi\colon \F(X,A)\to \Omega(X,A)$ defined by $\varphi(f)=f(X)$ for every $f\in \F(X,A)$. By \ref{coro_subseteq_iff_f1_leq_f2}, $\varphi$ and its inverse are order-preserving functions. The result follows.
\end{proof}

	Let $X$ be a finite T$_0$--space and let $f\colon X\to X$ be a continuous function such that $f\leq \id_X$. Since $f\geq f^{2}\geq f^{3}\geq\ldots$ and $X^{X}$ is finite, there exists $N\in\N$ such that $f^{N+1}=f^{N}$. It is clear that $f^{n}=f^{N}$ for every $n\geq N$. This motivates the following definition.
\begin{definition}
	Let $X$ be a finite T$_0$--space and let $f\colon X\to X$ be a continuous function such that $f\leq\id_X$. We define $f^{\infty}$ by $f^{N}$ where $N\in\N$ is such that $f^{N}=f^{N+1}$. 
\end{definition}

The following lemma states some simple properties of the construction of the previous definition.

\begin{lemma} \label{lemma_f_infty}
Let $X$ be a finite T$_0$--space and let $f,g\colon X\to X$ be continuous functions such that $f\leq \id_X$ and $g\leq \id_X$. Then:
\begin{enumerate}
\item $f^{\infty}\leq f\leq \id_X$.
\item $f^{\infty}\circ f^{\infty}=f^{\infty}$.
\item $f^{\infty}(X)$ is a dbp--retract of $X$.
\item For all $x\in X$, $x\in f^{\infty}(X)$ if and only if $f(x)=x$.
\item If $f\leq g$ then $f^{\infty}\leq g^{\infty}$.
\item $(fg)^{\infty}(X)=f^{\infty}(X)\cap g^{\infty}(X)$.
\end{enumerate}
\end{lemma}

\begin{proof}
The first two items follow easily from the definition of $f^{\infty}$. The third item follows from items (1) and (2) and theorem \ref{theo_dbpr_equivalences}. The proof of items (4) and (5) are easy and will be omitted.

Now we will prove (6). Since $fg\leq f$, $(fg)^\infty\leq f^\infty$. Thus, $(fg)^\infty(X)\subseteq f^\infty(X)$ by \ref{coro_subseteq_iff_f1_leq_f2}. Similarly, $(fg)^\infty(X)\subseteq g^\infty(X)$. Hence $(fg)^\infty(X)\subseteq f^\infty(X) \cap g^\infty(X)$. Now, if $x\in f^\infty(X) \cap g^\infty(X)$ then $f(x)=x=g(x)$ by item (4). Thus, $x=(fg)^\infty(x)\in (fg)^\infty(X)$.
\end{proof}

\begin{prop}
	\label{prop_F_closed_by_fginfty}
	Let $X$ be a finite T$_0$--space and let $A\subseteq X$.
	
	Then, for every $f,g\in \F(X,A)$:
	\begin{enumerate}
		\item $(fg)^{\infty}\in \F(X,A)$, and
		\item $(fg)^{\infty}(X)=f(X)\cap g(X)$.
	\end{enumerate}
\end{prop}

\begin{proof}
Let $f,g\in \F(X,A)$. Note that $f^\infty=f$ and $g^\infty=g$. Thus
\begin{displaymath}
A \subseteq f(X) \cap g(X) = f^\infty(X) \cap g^\infty(X) = (fg)^\infty(X)
\end{displaymath}
by item (6) of \ref{lemma_f_infty}. The result follows from items (1) and (2) of \ref{lemma_f_infty}.
\end{proof}

\begin{prop}
	Let $X$ be a finite T$_0$--space and let $A\subseteq X$.
		
	Then, $\Omega(X,A)$ is closed under intersections. In particular, $\Omega(X,A)$ has a minimum.
	\label{prop_minimum_dbpr}
\end{prop}
\begin{proof}
	Let $W_1,W_2\in\Omega(X,A)$. Let $\varphi$ be as in the proof of \ref{prop_Omega_iso_F} and let $f_k=\varphi^{-1}(W_k)$ for $k=1,2$. By \ref{prop_F_closed_by_fginfty}, we have that $W_1\cap W_2=f_1(X)\cap f_2(X)=(f_1f_2)^{\infty}(X)\in \Omega(X,A)$. 
\end{proof}
\begin{coro}
	Let $X$ be a finite T$_0$--space and let $A$ and $B$ be two dbp--retracts of $X$. Then $A\cap B$ is a dbp--retract of $X$. 
\end{coro}
\begin{rem}
	\label{rem_algorithm_dbpr}
	Let $X$ be a finite T$_0$--space and let $A$ be a subspace of $X$. Proposition \ref{prop_minimum_dbpr} implies that $A$ is a dbp--retract of $X$ if and only if $A$ is the minimum element of $\Omega(X,A)$. From \ref{prop_dbpr_of_Y_and_sub_of_X_is_dbpr_of_X} it follows that $A$ is a dbp--retract of $X$ if and only if any sequence of successive removals of down beat points of $X$ which do not belong to $A$ ends with the subspace $A$ when all such down beat points have been removed. Therefore, this gives an efficient algorithm to decide whether a subspace of a finite T$_0$--space $X$ is a dbp--retract of $X$.
	
	Observe that Theorem \ref{theo_Ux_cap_A} gives another efficient algorithm for doing the same. These two algorithms are esentially equivalent since after considering a linear extension of the partial order of $X-A$, the algorithm given by \ref{theo_Ux_cap_A} is equivalent to the identification (and removal) of down beat points in increasing order.
\end{rem}

\begin{ex} \label{ex_dbpr}
Let $X=\{a,b,c,d,e,f,g,h\}$ be the T$_0$--space that corresponds to the following Hasse diagram:
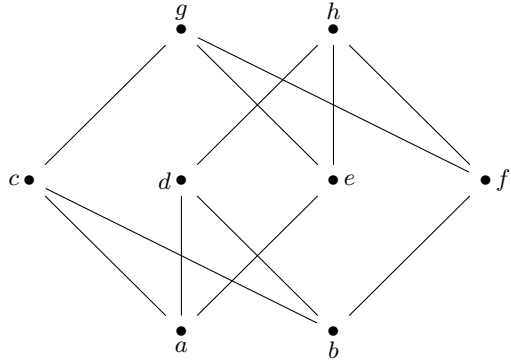
\begin{displaymath}
\begin{tikzpicture}[x=2cm,y=2cm]
\tikzstyle{every node}=[font=\footnotesize]
\draw (1,0) node(a){$\bullet$} node  [below] {$a$};
\draw (2,0) node(b){$\bullet$} node [below]{$b$};
\draw (0,1) node(c){$\bullet$} node [left]{$c$};
\draw (1,1) node(d){$\bullet$} node  [left] {$d$};
\draw (2,1) node(e){$\bullet$} node  [right] {$e$};
\draw (3,1) node(f){$\bullet$} node [right]{$f$};
\draw (1,2) node(g){$\bullet$} node [above]{$g$};
\draw (2,2) node(h){$\bullet$} node [above]{$h$};
\draw (g) -- (c);
\draw (g) -- (e);
\draw (g) -- (f);
\draw (h) -- (d);
\draw (h) -- (e);
\draw (h) -- (f);
\draw (c) -- (a);
\draw (c) -- (b);
\draw (d) -- (a);
\draw (d) -- (b);
\draw (e) -- (a);
\draw (f) -- (b);

\end{tikzpicture}
\end{displaymath}
Let $A=\{a,b,c,d\}$ and let $B=\{a,b,d,g\}$, both of them considered as subspaces of $X$.
We can obtain the subspace $A$ from $X$ by successively removing the down beat points $e$, $f$, $g$ and $h$.

Now, we can obtain $B$ by successively removing the beat points $e$, $f$, $h$ and $c$. Thus, $B$ is a strong deformation retract of $X$. However,
the set $\{a,b,c,d,g\}$ is minimal, with respect to set inclusion, among the dbp--retracts of $X$ that contain $B$. This dbp--retract is not equal to $B$, and hence $B$ is not a dbp--retract of $X$.
Observe that the same conclusion is achieved applying \ref{theo_Ux_cap_A} and noting that the set $U^{X}_c\cap B$ does not have a maximum element. 
\end{ex}

\section{Cofibrations between finite topological spaces}
\label{sec-cofibrations}

By \emph{cofibration} we will mean \emph{Hurewicz cofibration}, that is, a continuous function which has the homotopy extension property with respect to all topological spaces.

It is well known that every cofibration is a homeomorphism onto its image \cite{dieck2008algebraic,hatcher2002algebraic}.

A classic result states that if $X$ is a topological space and $A$ is a subspace of $X$ such that the inclusion of $A$ into $X$ is a cofibration, then $X\times \{0\}\cup A\times I$ is a retract of $X\times I$. The converse is easy to prove when $A$ is a closed subspace of $X$ and was proved without this assumption by Str\o{}m in \cite{Strom1968note} using the following lemma.

\begin{lemma}[Str\o{}m, {\cite[Lemma 3]{Strom1968note}}]
	\label{lemma_C_open}
	Let $X$ be a topological space and let $A$ be a subspace of $X$ such that $X\times \{0\}\cup A\times I$ is a retract of $X\times I$. Then, a subset $C$ of $X\times \{0\}\cup A\times I$ is open if and only if $C\cap X\times\{0\}$ is open in $X\times\{0\}$ and $C\cap A\times I$ is open in $A\times I$.
\end{lemma}
From this lemma, he obtained the following characterization of cofibrations.
\begin{theo}[Str\o{}m, {\cite[Theorem 2]{Strom1968note}}]
	\label{theo_strom}
	Let $X$ be a topological space and let $A$ be a subspace of $X$. Then, the inclusion $i\colon A\hookrightarrow X$ is a cofibration if and only if $X\times \{0\}\cup A\times I$ is a retract of $X\times I$.
\end{theo}

In this section, we will obtain a simple and combinatorial characterization of cofibrations between finite topological spaces and we will show how it is related to the notion of dbp--retracts of section \ref{sec_bpr}. To this end, we will give a simple alternative proof of Lemma \ref{lemma_C_open} in the case that the subspace $A$ is a finite space (Lemma \ref{lemma_C_open_finite_subspace}) from which our characterization of cofibrations between finite spaces will be obtained. 

Our proof of lemma \ref{lemma_C_open_finite_subspace} is based on Str\o{}m's proof of \ref{lemma_C_open}. Nevertheless, it is interesting to observe that, under the assumption that $A$ is a finite space, the hypothesis that $X\times \{0\}\cup A\times I$ is a retract of $X\times I$ of Lemma \ref{lemma_C_open} is not required for Lemma \ref{lemma_C_open_finite_subspace}. Moreover, by \ref{theo_strom}, this hypothesis holds if and only if the inclusion map $A\hookrightarrow X$ is a cofibration. Thus, from the characterization of cofibrations between finite spaces that will be given in this section, one can construct many examples of inclusion maps $A\hookrightarrow X$ for which \ref{lemma_C_open_finite_subspace} can be applied but \ref{lemma_C_open} can not.

\begin{lemma} \label{lemma_C_open_finite_subspace}
	Let $X$ be a topological space, let $A$ be a finite subspace of $X$. Then, a subset $C$ of $X\times\{0\}\cup A\times I$ is open in $X\times\{0\}\cup A\times I$ if and only if $C\cap X\times \{0\}$ is open in $X\times \{0\}$ and $C\cap A\times I$ is open in $A\times I$.
\end{lemma}
\begin{proof}
	Suppose that $C\cap X\times \{0\}$ is open in $X\times \{0\}$ and $C\cap A\times I$ is open in $A\times I$. Let $Y=X\times \{0\}\cup A\times I$. 
	Let $U=\{x\in X:(x,0)\in C\}$. It is clear that $U$ is open in $X$. Since $U\cap A$ is finite and $C\cap A\times I$ is open in $A\times I$, then there exists $\varepsilon>0$ such that $(U\cap A)\times [0,\varepsilon)\subseteq C$.
	
	It is easy to see that $C=\big(U\times [0,\varepsilon)\cap Y\big)\cup \big(C\cap A\times (0,1]\big)$.
	
	Furthermore, $U\times [0,\varepsilon)\cap Y$ is clearly open in $Y$. On the other hand, since $C\cap (A\times(0,1])$ is open in $A\times (0,1]$ and $A\times (0,1]$ is open in $Y$, it follows that $C\cap(A\times (0,1])$ is open in $Y$. Thus, $C$ is open in $Y$.
	
	The converse is clear.
\end{proof}

The following proposition follows easily from the previous lemma.

\begin{prop}
	Let $X$ and $Z$ be topological spaces and let $A$ be a finite subspace of $X$. Let $f\colon X\to Z$ and $H\colon A\times I \to Z$ be two continuous functions such that $H(a,0)=f(a)$ for every $a\in A$. Then, the function $\phi\colon X\times\{0\}\cup A\times I\to Z$ defined by
	\begin{displaymath}
		\phi(x,t)=\begin{cases}f(x)&\text{ if $t=0$,}\\H(x,t)&\text{ if $x\in A$}\end{cases}
	\end{displaymath}
	is continuous.
	
	Equivalently, $X\times\{0\}\cup A\times I$ is the mapping cylinder of the inclusion map $A\hookrightarrow X$.
\end{prop}

From the previous proposition we obtain the following corollary which is a particular case of \ref{theo_strom} but was obtained with a much simpler proof.

\begin{coro}
	\label{coro_ret_gives_cof}
	Let $X$ be a topological space, let $A$ be a finite subspace of $X$ and let $i\colon X\times \{0\}\cup A\times I\to X\times I$ be the inclusion map. Then the inclusion $A\hookrightarrow X$ is a cofibration if and only if there exists a retraction $r$ of $i$. 
\end{coro}

The following is one of the main results of this article.
\begin{theo}
	Let $X$ be a connected finite topological space and let $A$ be a non-empty subspace of $X$. Then, the inclusion $i\colon A\hookrightarrow X$ is a cofibration if and only if there exists a retraction $r\colon X\to A$ of $i$ such that $ir\leq \id_X$. 
	\label{theo_cofibration}
\end{theo}
\begin{proof}
	Let $Y=X\times\{0\}\cup A\times I$ and let $\iota\colon Y\to X\times I$ be the inclusion. We will prove that $\iota$ has a retraction if and only if $i$ has a retraction $r$ such that $ir\leq\id_X$. Hence, the theorem will follow from \ref{coro_ret_gives_cof} (or \ref{theo_strom}).
	
	Suppose that $\iota$ has a retraction $\rho$. Let $p_X\colon X\times I\to X$ and $p_I\colon X\times I\to I$ the canonical projections. 
	For each $t\in I$, we have a continuous function $i_t\colon X\to X\times I$ defined by $i_t(x)=(x,t)$ for every $x\in X$. Note that $p_I \iota\rho i_t$ is a continuous function from $X$ to $I$, and hence, it is a constant map by \ref{lemma_cte}. Now, if $a\in A$, $p_I \iota\rho i_t(a)=p_I \iota\rho(a,t)=p_I(a,t)=t$. It follows that $p_I \iota\rho i_t(x)=t$ for every $x\in X$.
	
	Let $\phi\colon I\to X^{X}$ be the function induced by the map $p_X\iota\rho\colon X\times I \to X$. Note that $\phi$ is continuous by \cite[Lemma 1]{stong1966finite} and that $\phi(0)=\id_X$. Hence, $0\in \phi^{-1}(U_{\id_X})$ and thus, there exists $\varepsilon>0$ such that $\phi(\varepsilon)\leq \id_X$. It can be readily verified that $\phi(\varepsilon)(x)\in A$ for every $x\in X$. Hence, we can restrict $\phi(\varepsilon)$ to a map  $r\colon X\to A$, and it is clear that $ir=\phi(\varepsilon)\leq \id_X$. And since $\iota\rho$ is the identity on $A\times I$, it easily follows that $ri=\id_A$.

	For the converse, suppose that there is a retraction $r$ of $i$ such that $ir\leq \id_X$. Let $\alpha\colon I\to X^X$ be defined by
	\begin{displaymath}
	\alpha(t)=\begin{cases}\id_X&\text{if $t=0$,}\\ir &\text{if $t>0$.}\end{cases}
	\end{displaymath}
	Since $ir\leq \id_X$, it follows that $\alpha$ is a continuous map. Let $\alpha^\flat \colon X\times I \to X$ be the function induced by $\alpha$ and the exponential law. We obtain that $\alpha^\flat$ is continuous by \cite[Lemma 1]{stong1966finite}.
	
	As above, let $p_I\colon X\times I\to I$ be the canonical projection. Let $\beta\colon X\times I\to X\times I$ be the map induced by $\alpha^\flat$ and $p_I$. It is easy to verify that $\im \beta \subseteq Y$. Let $\rho \colon X\times I\to Y$ be the range restriction of $\beta$ to $Y$. It is not difficult to check that $\rho$ is a retraction of $\iota$.
\end{proof}

\begin{rem} \label{rem_connected_components}
Let $X$ be a finite topological space and let $A$ be a subspace of $X$. Observe that, in order to determine whether the inclusion map $A\hookrightarrow X$ is a cofibration, we may always reduce our analysis to cases in which the hypotheses of the previous theorem are fulfilled. Indeed, since a finite topological space is the coproduct of its connected components, it follows that the inclusion map $A\hookrightarrow X$ is a cofibration if and only if the inclusion map $A\cap C \hookrightarrow C $ is a cofibration for each connected component $C$ of $X$. In addition, since the inclusion of the empty subspace in any topological space is a cofibration, we obtain that the inclusion map $A\hookrightarrow X$ is a cofibration if and only if the inclusion map $A\cap C \hookrightarrow C $ is a cofibration for each connected component $C$ of $X$ such that $C\cap A\neq \varnothing$.
\end{rem}

Theorem \ref{theo_cofibration} yields the unexpected results \ref{coro_cofibr_is_sdr}, \ref{prop_cofibr_closed_subspace} and \ref{coro_closed_cofibr_is_homeo}.

\begin{coro} \label{coro_cofibr_is_sdr}
	Let $X$ be a connected finite space and let $A$ be a non-empty subspace of $X$. If the inclusion $i\colon A\to X$ is a cofibration, then $A$ is a strong deformation retract of $X$.
\end{coro}
\begin{prop} \label{prop_cofibr_closed_subspace}
	Let $X$ be a connected finite space, let $A$ be a non-empty closed subspace of $X$. If the inclusion $i\colon A\to X$ is a cofibration, then $A=X$.
\end{prop}
\begin{proof}
	Suppose that $i$ is a cofibration. By \ref{theo_cofibration}, there exists $r\colon X\to A$ such that $ri=\id_A$ and $ir\leq \id_X$.
	Now let $x\in X$. We have that $x\geq r(x)\in A$. Since $A$ is closed, $x\in A$. The result follows. 
\end{proof}
\begin{coro} \label{coro_closed_cofibr_is_homeo}
	A closed cofibration between non-empty connected finite spaces is a homeomorphism.
\end{coro}
Recall that a pointed space $(X,x_0)$ is said to be \emph{well-pointed} if the inclusion $\{x_0\}\hookrightarrow X$ is a cofibration.  
\begin{prop}
	Let $(X,x_0)$ be a pointed connected finite space. Then $(X,x_0)$ is well-pointed if and only if $x_0\leq x$ for every $x\in X$. In particular, if $X$ is a T$_0$ space, $(X,x_0)$ is well-pointed if and only if $x_0$ is the minimum of $X$.
\end{prop}
\begin{proof}
	By \ref{theo_cofibration}, the space $(X,x_0)$ is well-pointed if and only if the only map $r\colon X\to \{x_0\}$ satisfies that $r(x)\leq x$ for every $x\in X$.
\end{proof}

In the case of connected finite T$_0$--spaces, non-trivial cofibrations are essentially dbp--retracts as the following result and its corollaries state.

\begin{prop}
	Let $X$ be a connected finite T$_0$--space and let $A$ be a non-empty subset of $X$. Then, the inclusion $i\colon A\to X$ is a cofibration if and only if $A$ is a dbp--retract of $X$.
	\label{coro_cof_iff_dbpr}
\end{prop}

\begin{proof}
	Immediate from \ref{theo_dbpr_equivalences} and \ref{theo_cofibration}.
\end{proof}

\begin{coro}
	\label{coro_cof_iff_fX_dbpr_Y}
	Let $X$ and $Y$ be finite T$_0$--spaces such that $Y$ is connected and $X\neq\varnothing$. Let $f\colon X\to Y$ be any function. Then $f$ is a cofibration if and only if $f$ is a homeomorphism onto its image and $f(X)$ is a dbp--retract of $Y$.
\end{coro}
\begin{coro}
	Let $X$ be a connected finite T$_0$--space. Then, every cofibration $f\colon A\to X$ with $A\neq\varnothing$ is a homeomorphism if and only if $X$ does not have down beat points. 
\end{coro}

We will prove now that in order to determine if a map between finite spaces is a cofibration we can always reduce our analysis to maps between finite T$_0$--spaces. To this end, we need some simple results which are contained in the following remark.

\begin{rem}\ \label{rem_T0-ization}
 Let $X$ be a finite topological space and let $q_X\colon X\to X_0$ be the quotient map.
\begin{enumerate}
 \item  The set of connected components of $X_0$ is $\{q_X(C)\tq C \textnormal{ is a connected component of $X$} \}$.
 \item Let $A$ be a subspace of $X$ and let $i\colon A\to X$ be the inclusion map. Then the map $i_0\colon A_0\to X_0$ is an embedding and $A_0$ is canonically homeomorphic to $q_X(A)$. In addition, $q_X(A\cap C)=q_X(A)\cap q_X(C)$ for each connected component $C$ of $X$.
\end{enumerate}
\end{rem}

The following proposition allows the reduction to the case of finite T$_0$--spaces in the characterization of cofibrations.

\begin{prop} \label{prop_i_cof_iff_i0_cof}
Let $X$ be a finite topological space and let $A$ be a subspace of $X$. Then, the inclusion $i\colon A\to X$ is a cofibration if and only if the map $i_0\colon A_0\to X_0$ is a cofibration.
\end{prop}

\begin{proof}
It is easy to prove that the map $i_0\colon A_0\to X_0$ is a retract of the inclusion $i\colon A\to X$. Thus, if $i$ is a cofibration then $i_0$ is also a cofibration.

For the converse, suppose first that $X$ is connected and $A\neq\varnothing$.

Let $q_A\colon A\to A_0$ and $q_X\colon X\to X_0$ be the canonical quotient maps and let $j_A\colon A_0\to A$ be a section of $q_A$. 
By \ref{theo_cofibration}, there exists a retraction $\rho$ of $i_0$ such that $i_0\rho\leq \id_{X_0}$. We define $r\colon X\to A$ by
\begin{displaymath}
r(x)=\begin{cases}x&\text{if $x\in A$,}\\j_A\rho q_X(x)&\text{if $x\in X-A$.}\end{cases}
\end{displaymath}

Observe that $j_A\rho q_X(a)\sim a$ for every $a\in A$. It follows that 
\begin{displaymath}
	r(a)=a\leq j_A\rho q_X(a)\leq j_A\rho q_X(x)=r(x)
\end{displaymath}
for every $a\in A$ and every $x\in X-A$ such that $a\leq x$, and that 
\begin{displaymath}
r(x)=j_A\rho q_X(x)\leq j_A\rho q_X(a)\leq a=r(a)
\end{displaymath}
for every $x\in X-A$ and every $a\in A$ such that $x\leq a$. Hence, $r$ is continuous.

It is clear that $ri=\id_A$. 
On the other hand, $ir(x)=x$ for every $x\in A$ and, since $q_Xi=i_0q_A$, then
\begin{displaymath}
q_X ir(x)=q_Xij_A\rho q_X(x)=i_0q_Aj_A\rho q_X(x)=i_0\rho q_X(x)\leq q_X(x)
\end{displaymath}
for every $x\in X-A$. Hence, $ir\leq \id_X$. Then, $i$ is a cofibration by \ref{theo_cofibration}.

The general case follows applying \ref{rem_connected_components} and \ref{rem_T0-ization}.
\end{proof}

The following result follows easily from \ref{prop_i_cof_iff_i0_cof}

\begin{prop}\label{prop_f_cof_iff_subesp_and_f0_cof}
Let $X$ and $Y$ be finite topological spaces and let $f\colon X\to Y$ be a continuous map. Then $f$ is a cofibration if and only if $f$ is an embedding and $f_0\colon X_0\to Y_0$ is a cofibration.
\end{prop}

Note that a map $f\colon X\to Y$ between finite spaces might not be a cofibration even if $f_0\colon X_0\to Y_0$ is a cofibration. For example, let $X=\{0,1\}$ with the indiscrete topology, let $Y$ be the singleton and let $f\colon X\to Y$ be the only possible map. Then the map $f\colon X\to Y$ is not a cofibration since it is not injective but the map $f_0\colon X_0\to Y_0$ is a homeomorphism.

\begin{rem}\label{rem_algorithm_cofibr}
Combining some of the results developed above we obtain a simple algorithm for determining whether a function between finite topological spaces is a cofibration, which is described below.

Let $X$ and $Y$ be finite topological spaces and let $f\colon X\to Y$ be a function. Clearly $f$ is a cofibration if and only if the range restriction $f|^{f(X)}\colon X\to f(X)$ is a homeomorphism and the inclusion map $f(X)\hookrightarrow Y$ is a cofibration. Observe that $f|^{f(X)}$ is a homeomorphism if and only if for all $x_1,x_2\in X$, $x_1\leq x_2 \Leftrightarrow f(x_1)\leq f(x_2)$.

Let $A=f(X)$. By \ref{rem_T0-ization} and \ref{prop_i_cof_iff_i0_cof} the inclusion map $A\hookrightarrow Y$ is a cofibration if and only if the inclusion map $q_Y(A)\hookrightarrow Y_0$ is a cofibration. By \ref{rem_connected_components}, this holds if and only if the inclusion map $q_Y(A)\cap C \hookrightarrow C $ is a cofibration for each connected component $C$ of $Y_0$ such that $q_Y(A)\cap C\neq\varnothing$. Now, for each connected component $C$ of $Y_0$ such that $q_Y(A)\cap C\neq\varnothing$, the inclusion map $q_Y(A)\cap C \hookrightarrow C $ is a cofibration if and only if $q_Y(A)\cap C$ is a dbp--retract of $C$ by \ref{coro_cof_iff_dbpr}.

Therefore, we obtain that the inclusion map $A\hookrightarrow Y$ is a cofibration if and only if $q_Y(A)\cap C$ is a dbp--retract of $C$ for each connected component $C$ of $Y_0$ such that $q_Y(A)\cap C\neq\varnothing$. This condition can be verified algorithmically, as was noted in \ref{rem_algorithm_dbpr}.
\end{rem}

As an example of application of the previous results, we will determine if certain inclusion maps regarding mapping cylinders are cofibrations. Recall that if $X$ and $Y$ are topological spaces and $f\colon X\to Y$ is a continuous map then the inclusion maps of $X$ and $Y$ into the mapping cylinder of $f$ are cofibrations.

Since the mapping cylinder of a continuous function between (non-empty) finite topological spaces is not finite, we can not apply our results to the standard mapping cylinder. However, we are interested in the discrete analog of the mapping cylinder for continuous maps between finite T$_0$--spaces which is more suitable for working in the finite setting (see \cite{barmak2011algebraic}).

\begin{definition}
	Let $X$ and $Y$ be finite T$_0$--spaces and let $f\colon X\to Y$ be a continuous function. The \emph{non-Hausdorff mapping cylinder} of $f$ is the space $B(f)$ whose underlying set is $X\coprod Y$ and whose topology is induced by the following order in $B(f)$:
	\begin{displaymath}
		z\leq z'\text{ in $B(f)$ if and only if }
		\begin{cases}
			z\leq z'\text{ in $X$}&\text{ if $z,z'\in X$,}\\
			f(z)\leq z'\text{ in $Y$}&\text{ if $z\in X$ and ${z'\in Y}$, or}\\
			z\leq z'\text{ in $Y$}&\text{ if $z,z'\in Y$.}
		\end{cases}  
	\end{displaymath}
	The canonical inclusion maps of $X$ and $Y$ into $B(f)$ will be denoted by $j_X$ and $j_Y$, respectively.
\end{definition}
\begin{rem}
	Let $f\colon X\to Y$ be a continuous function between non-empty finite T$_0$--spaces. Since $Y$ is closed in $B(f)$, from \ref{coro_closed_cofibr_is_homeo} it follows that $j_Y$ is not a cofibration.
\end{rem}

\begin{prop}
	Let $X$ and $Y$ be finite T$_0$--spaces and let $f\colon X\to Y$ be a continuous function. Then $j_Y^{\op}$ is a cofibration.
\end{prop}

\begin{proof}
	As in the proof of Lemma 2.8.2 of \cite{barmak2011algebraic},
	let $j=j_Y\colon Y\to B(f)$ be the inclusion map and let $r\colon B(f)\to Y$ defined by 
	\begin{displaymath}
		r(z)=\begin{cases}f(z)&\text{ if $z\in X,$}\\z&\text{ if $z\in Y.$}\end{cases} 
	\end{displaymath}
	If $x\in X$ and $y\in Y$ are such that $x\leq y$ in $B(f)$, then $f(x)\leq y$ and hence $r(x)=f(x)\leq y=r(y)$. Therefore, $r$ is continuous.
	It is clear that $rj=\id_Y$.
	
	Now, $jr(x)=f(x)\geq x$ in $B(f)$ for every $x\in X$. Thus, $jr\geq\id_{B(f)}$. The result follows from \ref{theo_cofibration}.
\end{proof}

The following result is already present in the proof of Proposition 4.6.6 of \cite{barmak2011algebraic} with a different terminology. We give here a proof of it using our tools.

\begin{prop} \label{prop_X_dbpr_Bf}
	Let $X$ and $Y$ be finite T$_0$--spaces and let $f\colon X\to Y$ be a continuous function. Then $X$ is a dbp--retract of $B(f)$ if and only if $f^{-1}(U^{Y}_y)$ has a maximum for every $y\in Y$.
\end{prop}
\begin{proof}
Note that $f^{-1}(U^{Y}_y)=U^{B(f)}_y\cap X$ for every $y\in Y$. The result follows from \ref{theo_Ux_cap_A}.
\end{proof}

\begin{coro}
Let $X$ and $Y$ be finite T$_0$--spaces such that $Y$ is connected and $X\neq\varnothing$. Let $f\colon X\to Y$ be a continuous map. Then the inclusion map $j_X\colon X \to B(f)$ is a cofibration if and only if $f^{-1}(U^{Y}_y)$ has a maximum for every $y\in Y$.
\end{coro}

\begin{proof}
By \ref{coro_cof_iff_dbpr}, the map $j_X\colon X \to B(f)$ is a cofibration if and only if $X$ is a dbp--retract of $B(f)$. The result follows from \ref{prop_X_dbpr_Bf}.
\end{proof}

\bibliographystyle{acm}
\bibliography{references}

\end{document}